\documentclass{amsart}
\usepackage{amssymb}

\newtheorem{theorem}{Theorem}[section]
\newtheorem{proposition}[theorem]{Proposition}
\newtheorem{lemma}[theorem]{Lemma}
\newtheorem{corollary}[theorem]{Corollary}

\theoremstyle{definition}
\newtheorem{definition}[theorem]{Definition}
\newtheorem{question}[theorem]{Question}
\newtheorem{remark}[theorem]{Remark}
\newtheorem{example}[theorem]{Example}

\DeclareMathOperator{\dist}{dist}
\DeclareMathOperator{\ran}{ran}
\DeclareMathOperator{\rk}{rank}

\let\sectionmark\S

\newcommand{\C}{\mathcal{C}}
\newcommand{\D}{\mathcal{D}}
\newcommand{\F}{\mathcal{F}} 
\renewcommand{\H}{\mathcal{H}} 
\newcommand{\I}{\mathcal{I}} 
\newcommand{\J}{\mathcal{J}} 
\newcommand{\K}{\mathcal{K}} 
\renewcommand{\L}{\mathcal{L}} 
\newcommand{\N}{\mathcal{N}} 
\renewcommand{\P}{\mathcal{P}} 
\newcommand{\R}{\mathcal{R}} 
\renewcommand{\S}{\mathcal{S}}
\newcommand{\T}{\mathcal{T}}
\newcommand{\U}{\mathcal{U}} 
\newcommand{\NN}{\mathbb{N}} 

\newcommand{\bh}{B(\H)}
\newcommand{\dn}{\D(\N)} 
\newcommand{\dnn}{\D(\NN)} 
\newcommand{\e}{\epsilon}

\newcommand{\ip}[1]{\langle #1\rangle}
\newcommand{\qt}{\mathcal{QT}}

\newcommand{\tn}{\T(\N)} 
\newcommand{\tnn}{\T(\NN)} 

\newcommand{\pint}{\Pi_\text{int}}
\newcommand{\pmax}{\Pi_\text{max}}
\newcommand{\pmin}{\Pi_\text{min}}

\newcommand{\types}{Type-S}

\begin{document}

\title{On the primitive ideals of nest algebras}
\author{John Lindsay Orr}
\address{
  Toll House,
  Traquair Road,
  Innerleithen, EH44 6PF,
  United Kingdom
}
\email{me@johnorr.us}
\begin{abstract}
  We show that Ringrose's diagonal ideals are primitive ideals in a nest
  algebra (subject to the Continuum Hypothesis). This provides for the
  first time concerete descriptions of enough primitive ideals to obtain the
  Jacobson radical as their intersection. Separately, we provide a
  standard form for all left ideals of a nest algebra, which leads to
  insights into the maximal left ideals. In the case of atomic nest algebras
  we show how primitive ideals can be categorized by their behaviour on
  the diagonal, and provide concrete examples of all types.
\end{abstract}

\maketitle

\section{Introduction}

The Jacobson radical has been a frequent object of study in non-selfadjoint
algebras, and considerable effort has been expended to identify the radical
in the context of various classes of non-selfadjoint algebras, e.g.,
\cite{
  Ringrose:OnSoAlOp,
  Peters:SePrCstarAl,
  DavidsonOrr:JaRaCSL,
  Donsig:SeTrAfAl,
  MastrangeloMohlySolel:LoRaTrOpAl,
  DavidsonKatsoulisPitts:StFrSeAl,
  DonsigKatavolosManoussos:JaRaANCrPr,
  KatsoulisRamsey:CrPrOpAl}.
Why is this? At fist glance it might seem
that since many non-selfadjoint algebras are modelled more or less on the
algebra of finite-dimensional upper triangular matrices, the desire is
to obtain Wedderburn-type structure theorems for the algebras.
In fact, however, the Jacobson radical is rarely the right ideal for
such a decomposition, if it is even possible. The Jacobson radical
is often too small, and indeed in some cases non-selfadjoint algebras
are even semisimple
\cite{
  Donsig:SeTrAfAl, 
  MastrangeloMohlySolel:LoRaTrOpAl,
  DonsigKatavolosManoussos:JaRaANCrPr}.
Thus knowledge about the Jacobson radical rather points towards
more general structural information about the algebra and,
in particular, when the radical is small, indicates the presence
of a rich supply of irreducible representations, even in algebras
which have a strong heuristic connection with the upper triangular
matrix algebra.

The nest algebras are one such case. Indeed the main result of
Ringrose's paper \cite{Ringrose:OnSoAlOp}, which introduced the
class of nest algebras, was to describe the Jacobson radical $\R_\N$ of a
nest algebra $\tn$ (see Section~\ref{preliminaries-sect} below for precise
definitions of terms). However, except in the trivial case of a finite
nest, there is no Wedderburn-type decomposition $\tn = \D(\N)\oplus\R_\N$
as the sum of the diagonal algebra and the Jacobson radical. In fact
by \cite[Theorem 4.1]{Orr:TrAlIdNeAl}, a decomposition $\tn = \D(\N) \oplus \R$
for some ideal $\R$ is only possible if $\R$ is Larson's ideal $\R^\infty_\N$
\cite{Larson:NeAlSiTr}, and then only if the nest has no continuous part.
At issue here is the fact that unless the nest is finite $\R^\infty_\N$
is much bigger than the Jacobson radical;
in the case of upper triangular matrixes on $\ell^2(\NN)$, $\R^\infty_\N$ is
the collection of all strictly upper triangular operators, while
$\R_\N$ is the set of \emph{compact} strictly upper triangular operators.
Thus, the comparatively small Jacobson radical in nest algebras indicates that
there must be many irreducible representations other than the trivial ones
obtained as the compression to an atom of the nest.

However, up to now, the only other primitive ideals which could be identified explicitly
were the maximal two-sided ideals. (Maximal two-sided ideals are primitive; see
Remark~\ref{prim-ideals-remark} for a review of this and other ring-theoretic
facts.) In \cite{Orr:MaIdNeAl} we described the maximal two-sided ideals
of a continuous nest algebra and in \cite{Orr:MaTwSoIdNeAl} we extended the
description to cover all nest algebras.
(It should be noted these results rest on deep foundations; between them, they
require the similarity theory of nests and the Paving Theorem.) Even so, however,
these ideals alone do not account for the small Jacobson radical. Their intersection,
called the \emph{strong radical}, is similar in character to $\R^\infty_\N$
and in fact the two coincide when the nest is atomic.

The goal of this paper is to identify enough examples of primitive ideals of 
nest algebras to account for the small Jacobson radical, by which we mean that
their intersection should equal the Jacobson radical. The key examples
have been in plain view all along; they are the ``diagonal ideals'' which
Ringrose used in his original description of the radical
\cite[Theorem 5.3]{Ringrose:OnSoAlOp}. We shall show in
Theorem~\ref{diagonal-are-primitive-thm} that the diagonal ideals
are primitive.
This answers an open question of Lance~\cite{Lance:SoPrNeAl} (repeated
in~\cite{Davidson:NeAl}).
Interestingly, this result
relies on assuming a positive answer to the Continuum Hypothesis. See
the excellent survey paper \cite{Weaver:SeThCstarAl} for other recent
results in operator algebras which make use of nonstandard foundational
considerations.

After this, we turn to an analysis of the left ideals of
nest algebras in Section~\ref{left-ideals-sect}. We establish a standard form for 
\emph{all} left ideals, and also a stronger form which holds for many
norm-closed left ideals,
including the maximal left ideals. In Section~\ref{atomic-nest-algebras-sect}
we explore the primitive ideals of atomic nest algebras in more depth.
We identify three classes of
primitive ideals (the smallest, the largest, and the intermediate ones),
and we show that they are distinguished by their behaviour on the diagonal.
Section~\ref{infinite-upper-sect} focusses on the infinite upper-triangular
matrices, where we can give concrete examples of all types of primitive ideals,
and also applications to quasitriangular algebras.

\section{Preliminaries}\label{preliminaries-sect}

Throughout this paper the underlying Hilbert spaces are always assumed separable.
A \emph{nest} is a set of projections on a Hilbert space which is linearly ordered,
contains 0 and I, and is weakly closed (or, equivalently, order-complete). The
\emph{nest algebra}, $\tn$, of a nest $\N$ is the set of bounded operators leaving
invariant the ranges of $\N$.
The \emph{diagonal algebra}, $\dn$, is the set of operators having the ranges
of projections in $\N$ as reducing subspaces; equivalently, the commutant of $\N$.
An \emph{interval} of $\N$ is the difference $N-M$ of two 
projections $N > M$ in $\N$. Minimal intervals are called \emph{atoms} and the atoms
(if there are any) are pairwise orthogonal. If the join of the atoms is $I$ the nest
is called \emph{atomic};  if there are no atoms it is called \emph{continuous}.
For $N\in\N$, define
\[
  N^- := \bigvee \{M\in\N : M < N\}
  \quad\text{and}\quad
  N^+ := \bigwedge \{M\in\N : M > N\}
\]
Conventionally $0^- = 0$ and $I^+ = I$.
If $N > N^-$ then $N - N^-$ is an atom of $N$, and all atoms are of this form.
Conversely, if $N = N^- > 0$ then there is a strictly increasing sequence of projections
in $\N$ which converge to $N$. Similar remarks apply for $N^+$.
We shall make continual use of the fact that the rank-1 operator
$x\mapsto \langle x, f\rangle e$, which we write as $ef^*$, belongs to
$\tn$ if and only if there is an $N\in\N$ such that $e\in\ran(N^+)$
and $f\in\ran(N^\perp)$.
See \cite{Davidson:NeAl} for further properties of nest algebras.

\begin{example}
Let $\H := \ell^2(\NN)$ and let $\{e_i\}_{i=1}^\infty$ be the standard basis.
For $n\in\NN$, let $N_n$ be the projection onto the span of $\{e_1,\ldots, e_n\}$ and
let $\N := \{N_n : n\in\NN\}\cup\{I\}$. This is a nest, and $\tn$ is the
nest algebra of all infinite upper triangular operators with respect to the
standard basis. By slight abuse of notation, we write $\tnn$ for this algebra.
\end{example}

We now recall Ringrose's description of the Jacobson radical of a nest algebra,
in terms of diagonal seminorms and diagonal ideals:

\begin{definition}
Let $\N$ be a nest and fix $N<I$ in $\N$. The \emph{diagonal seminorm function}
$i^+_N(X)$ is defined for $X\in\tn$ by
\[
  i^+_N(X) := \inf \{\|(M-N)X(M-N) : M > N \text{ in } \N \}
\]
Likewise, for $N>0$ the diagonal seminorm function $i^-_N(X)$ is
\[
  i^-_N(X) := \inf \{\|(N-M)X(N-M) : M < N \text{ in } \N \}
\]
\end{definition}

It is straightforward to see that the functions $i^\pm_N$ are submultiplicative
seminorms on $\tn$ and dominated by the norm, and so their kernels
are norm closed two-sided ideals of $\tn$:

\begin{definition}
Let $\N$ be a nest. The \emph{diagonal ideals} are the ideals
\[
  \I_N^+ := \{X\in\tn : i_N^+(X) = 0 \} \qquad\text{(for $N<I$)}
\]
and
\[
  \I_N^- := \{X\in\tn : i_N^-(X) = 0\} \qquad\text{(for $N>0$)}
\]
\end{definition}

The diagonal ideals can be viewed as generalizations of those ideals of upper-triangular
$n\times n$ matrices consisting of all the matrices which vanish at a particular
diagonal entry. Indeed if $N > N^-$ then
\begin{equation}\label{inminus-at-an-atom-equ}
  \I_N^- = \{X\in\tn : (N - N^-)X(N - N^-) = 0\}
\end{equation}
However if $N = N^-$, then $I^-_N$ is the set of operators \emph{asymptotically}
vanishing close to $N$ (from below). More precisely, in the case of $\tnn$,
$\I_N^-$ is of the form (\ref{inminus-at-an-atom-equ}) for all $N<I$
and $I_I^-$ is the compact operators of $\tnn$. See Section~\ref{infinite-upper-sect}
for a detailed discussion of the primitive ideals in this algebra.

Ringrose gave the following description of the Jaconson radical in terms of
these diagonal ideals.

\begin{theorem}[\cite{Ringrose:OnSoAlOp}, Theorem 5.3]
The Jacobson radical of $\tn$ is the intersection of the diagonal ideals
of $\tn$.
\end{theorem}

A key point to bear in mind is that although the diagonal ideals
are related to the primitive ideals, as the next result quoted shows,
they were not known to be primitive. Lance \cite{Lance:SoPrNeAl}
asked whether the diagonal ideals are primitive and, in his study of
the diagonal ideals and their quotients and proved a number of results which
are entailed by primitivity. In Theorem~\ref{diagonal-are-primitive-thm}
we show that the diagonal ideals are in fact primitive ideals.

The following useful result shows that each primitive ideal of a nest algebra
is associated with a unique diagonal ideal.

\begin{theorem}[\cite{Ringrose:OnSoAlOp}, Theorem 4.9]\label{prim-contains-diagonal-thm}
Every primitive ideal of $\tn$ contains exactly one diagonal ideal.
\end{theorem}

Based on this result we adopt the following notation:

\begin{definition}\label{isubp-def}
If $\P$ is a primitive ideal of the nest algebra $\tn$, write
$\I_\P$ for the unique diagonal ideal contained in $\P$.
\end{definition}

Finally we close the section by recalling Larson's ideal~\cite{Larson:NeAlSiTr},
$\R_\N^\infty$:

\begin{definition}
Let $\R_N^\infty$ be the set of $X\in\tn$ such that, given $\e>0$ we can find
a collection $\{N_i - M_i : \in\NN \}$ of pairwise orthogonal intervals of $\N$
which sum to $I$ and such that $\|(N_i - M_i)X(N_i - M_i)\| < \e$.
\end{definition}

Ringrose~\cite[Theorem 5.4]{Ringrose:OnSoAlOp} provides an alternate description
of the Jacobson radical which is formally very similar to Larson's ideal. The
only difference is the requirement that the collections of pairwise orthogonal
intervals must be \emph{finite}. However this makes an enormous difference to the
size of the ideal as the following example shows.

\begin{example}
Let $\N$ be the canonical nest on $\ell^2(\NN)$. Then $\R_\N$ is the set of
zero-diagonal compact operators in $\tnn$ and $\R_\N^\infty$ is the set
of all zero-diagonal operators in $\tnn$. Note in particular that
$\tnn = \dnn \oplus \R_\N^\infty$ but that $\tnn \not= \dnn \oplus\R_\N$ (for
example, the right-hand side fails to contain the unilateral backward shift).
\end{example}

\section{The diagonal ideals are primitive}\label{diagonal-ideal-are-primitive-sect}

The main result of this section is Theorem~\ref{diagonal-are-primitive-thm},
in which we prove that the diagonal ideals of a nest algebra are primitive.
We start by recalling some basic
facts about primitive ideals which can be found in many standard texts of ring theory
or Banach algebras. See, e.g., \cite[Chapter III]{BonsallDuncan:CoNoAl}.

\begin{remark}\label{prim-ideals-remark}
Let $A$ be a unital Banach algebra. The (left) primitive ideals of $A$ are the annihilators
of left $A$-modules, or, equivalently, the kernels of
the irreducible representations of $A$. If $P$ is any primitive ideal of $A$ then
there is a maximal left ideal $L$ of $A$ such that $P$ is the kernel of the left
regular representation of $A$ on $A/L$. Thus $P$ is the largest two-sided ideal of
$A$ contained in $L$, and is equal to
\[
  \{ x\in A : xA \subseteq L \}
\]
From this, together with the maximality of $L$, it follows easily that $x\not\in P$
if and only if there are $a,b\in A$ such that $ e - axb \in L$ (where $e$ is the unit
of $A$). Finally, of course, the Jacobson radical is, by definition, the intersection
of all the primitive ideals of $A$. Analogously, the \emph{right} primitive
ideals are the kernels of right $A$-modules and each right primitive ideal
is the kernel of the right module action of $A$ on the quotient $A/R$ of $A$
by some maximal right ideal. The intersection of the maximal right primitive ideals
is also the (same) Jacobson radical.
\end{remark}

Lemma~\ref{block-diagonal-lemma} will enable us to convert arbitrary upper
triangular operators to block diagonal form. It relies on the following useful
technical lemma which we quote in full.

\begin{lemma}\cite[Lemma 2.2]{Orr:StIdNeAl}\label{interpolating-vectors-lemma}
Let $X\in\bh$ and let $P_n,Q_n$ ($n\in\NN$) be sequences of projections such that
$\dist(P_nXQ_n, \F_{4n-4}) > 1$ for all $n$, where $\F_k$ denotes the set of
operators of rank not greater than $k$. Then there are orthonormal sequences
$x_i\in P_i\H$ and $y_i\in Q_i\H$ such that $\ip{x_i, Xy_j} =0$ for all $i\not= j$,
and $\ip{x_i, Xy_i}$ is real and greater than $1$ for all $i\in\NN$.
\end{lemma}

\begin{lemma}\label{block-diagonal-lemma}
Suppose $X\in\tn$ but $X\not\in\I^-_N$ for some $N=N^->0$ in $\N$. Then there are $A,B\in\tn$
and a sequence $N_k$ of nest projections strictly increasing to $N$ such that
\[
  AXB = \sum_{k=1}^\infty (N_k - N_{k-1}) AXB (N_k - N_{k-1})
\]
and each of the terms
$(N_k - N_{k-1}) AXB (N_k - N_{k-1})$ has norm greater than $1$.
\end{lemma}

\begin{proof}
Rescaling if necessary, assume $i_N^-(X) > 1$.
Choose a sequence $N_k\in\N$ which increases strictly to $N$. We shall inductively
construct a subsequence $N_{k_n}$ such that
$\dist((N_{k_n} - N_{k_{n-1}})X(N_{k_n} - N_{k_{n-1}}), \F_{4n-4}) > 1$
for all $k$, and the result will follow from an easy application of
Lemma~\ref{interpolating-vectors-lemma}. Take $k_1 := 1$ and suppose
$k_1 < k_2 < \cdots < k_{n-1}$ to have been chosen with the desired property.

Suppose for a contradiction that
$\dist((N_k - N_{k_{n-1}})X(N_k - N_{k_{n-1}}), \F_{4n-4}) \le 1$ for all
$k>k_{n-1}$.
Fix an $a$ with $1 < a < i_N^-(X)$ and for each $k\ge k_{n-1}$ find $F_k\in\F_{4n -4}$ such that
\[
  \|(N_k - N_{k_{n-1}})X(N_k - N_{k_{n-1}}) - F_k\| < a
\]
The sequence $F_k$ is norm-bounded and so has a $w^*$-convergent subsequence,
$F_{m_j} \rightarrow F$. But $F\in\F_{4n-4}$ since $\F_{4n-4}$ is $w^*$-closed
and, by the the lower semicontinuity of the norm,
\begin{eqnarray*}
  i_N^-(X) &\le& \|(N - N_{k_{n-1}})X(N - N_{k_{n-1}}) - F\|     \\
           &\le& \liminf_{j\rightarrow\infty}
               \|(N_{m_j} - N_{k_{n-1}})X(N_{m_j} - N_{k_{n-1}}) - F_{m_j}\| \\
           &\le& a
\end{eqnarray*}
which is a contradiction. Thus we find $k_n>k_{n-1}$ with which to continue the induction.

With $N_{k_n}$ chosen, apply Lemma~\ref{interpolating-vectors-lemma} to obtain
unit vectors $x_n, y_n$ in the range of $N_{k_n} - N_{k_{n-1}}$ such that
$\ip{x_m, Xy_n} = 0$ for all $m \not= n$, and $\ip{x_m, Xy_m} > 1$ for all $m\in\NN$.
Set
\[
  A := \sum_{n=1}^\infty x_{3n}x^*_{3n+1}
  \qquad\text{and}\qquad
  B := \sum_{n=1}^\infty y_{3n+1}y^*_{3n+2}
\]
Then $A,B\in\T(\N)$ since the terms of both sums are of the form $N_{k_m}RN_{k_m}^\perp$, and
$AXB = \sum_{n=1}^\infty \langle x_{3n+1}, X y_{3n+1}\rangle x_{3n}y_{3n+2}^*$, so that
$AXB = \sum_{n=1}^\infty(N_{k_{3n+2}} - N_{k_{3n-1}})AXB(N_{k_{3n+2}} - N_{k_{3n-1}})$
and each of the terms of the sum has norm greater than $1$.
\end{proof}

The following, unfortunately rather technical, definition is central to
our analysis in this section.

\begin{definition}
Fix a nest $\N$ and a projection $N\in\N$.
Say that a set $\S$ of operators in $\bh$ are of \types\ if there exists a
strictly increasing sequence $N_n$ in $\N$ which converges to $N$, and a
sequence of unit vectors $x_n = (N_n - N_{n-1})x_n$
such that for each $X\in\S$ both $X x_n \rightarrow 0$ and $X^* x_n \rightarrow 0$.
\end{definition}

Clearly if $\S\subseteq\tn$ is of \types, then it lies in both a proper left ideal of $\tn$
and in a proper right ideal of $\tn$. Note, however that it need not lie in a proper
two-sided ideal; for example consider the singleton $\{I - U\}$ where $U$ is the unilateral
backward shift on $\ell^2(\NN)$. This is \types\ with respect to the sequences
$N_{2^n}$ and $x_n := 2^{\frac{1 - n}{2}}\sum_{i=2^{n-1}}^{2^n -1} e_i$
but does not lie in a proper two-sided ideal of $\tnn$.
In fact this example is the prototype of the analysis which follows and an analogous
sequence is at the heart of the proof of the next lemma. Note also that,
strictly speaking, ``\types'' is a property which a set
has with respect to a particular $\N$ and $N\in\N$. In the following arguments these will
always be easily discerned from the context.

\begin{lemma}\label{countable-extension-lemma}
Fix a nest $\N$ and a projection $N=N^->0$ in $\N$ and let $\{X_i:i\in\NN\}$ be a set of \types.
Let $X\in\tn$ but $X\not\in\I_N^-$.
Then there are $A, B\in\tn$ such that $\{I - AXB\} \cup \{X_i:i\in\NN\}$
is also of Type-S.
\end{lemma}

\begin{proof}
Take a sequence $N_n\in\N$ which increases strictly to $N$ and unit vectors
$x_n = (N_n - N_{n-1})x_n$ such that $X_i x_n, X^*_i x_n \rightarrow 0$ for all $i\in\NN$.

By Lemma~\ref{block-diagonal-lemma} there are $A,B$ in $\tn$ and a sequence of
nest projections strictly increasing to $N$ such that $AXB$ is block diagonal
with respect to these projections and each of the blocks has norm greater than $1$.
Since $N_k$ and $x_k$ demonstrate the \types\ property, so does
any subsequence of theirs and so, replacing $N_k$ with a subsequence, we may assume that
each interval $N_k - N_{k-1}$ dominates a block of $AXB$.
Multiplying $AXB$ by a diagonal projection to select only those blocks which are
dominated by an $N_k - N_{k-1}$ and have norm greater than $1$, and replacing
$X$ with the resulting operator we may now assume that $X$ is block diagonal with respect to $N_k$,
and that all the blocks $(N_k-N_{k-1})X(N_k-N_{k-1})$ have norm greater than $1$.

We shall inductively construct a new sequence of unit vectors
$y_n = (N_{k_n} - N_{k_{n-1}})y_n$ for a subsequence $(k_n)$, together with contractions
\[
  A_n = (N_{k_n} - N_{k_{n-1}})A_n(N_{k_n} - N_{k_{n-1}})
\]
and
\[
  B_n = (N_{k_n} - N_{k_{n-1}})B_n(N_{k_n} - N_{k_{n-1}})
\]
in $\tn$ such that
\[
  \max\{\|(I - A_n X B_n) y_n\|, \|(I - A_n X B_n)^* y_n\|\}<1/n
\] 
and
$\max\{\|X_i y_n\|, \|X^*_i y_n\|\} < 1/n$ for all $1\le i\le n$. 
The result will then follow by taking $A := \sum_{n=1}^\infty A_n$ and $B := \sum_{n=1}^\infty B_n$.

To perform the induction, fix $n$ and suppose $k_m$, $y_m$, $A_m$, and $B_m$ have
been chosen for all $m<n$. (To get the induction started when $n=1$, define $k_0 := 0$ and observe
that no other features of the preceding steps are used in the induction step which follows.)

Note that, for all sufficiently large $m$,
\[
  \max\{\|X_i x_m\|,\|X^*_i x_m\|\} < 1/(2n^2)
\]
for all $1\le i \le n$. Thus, taking $N=4n^2$ we can pick $m_1<m_2<\cdots<m_N$ such that
$m_1 > k_{n-1} + 1$, each $m_{j} > m_{j-1} + 1$, and for all $1\le i \le n$ and $1 \le j \le N$,
$\max\{\|X_i x_{m_j}\|, \|X^*_i x_{m_j}\|\} < 1/(2n^2)$.

Set $k_n := m_N$ and $y_n := N^{-1/2}\sum_{j=1}^N x_{m_j}$, which is a unit vector since
the $x_{m_j}$ are pairwise orthogonal. For each $1 < j\le N$, 
the interval $N_{m_j - 1} - N_{m_{j-1}}$ dominates a diagonal block of $X$
which has norm greater than 1. Thus, we can choose
vectors $e_j$ and $f_j$ in $N_{m_j - 1} - N_{m_{j-1}}$ with
$\|e_j\| \ge \|f_j\| = 1$ and $e_j = X f_j$ and set
\[
  A_n := \sum_{j=2}^N \|e_j\|^{-1}\, x_{m_{j-1}}e_j^*
  \qquad\text{and}\qquad
  B_n := \sum_{j=2}^N \|e_j\|^{-1}\,f_jx_{m_j}^*
\]
Since
\[
  x_{m_{j-1}}e_j^* = N_{m_{j-1}}(x_{m_{j-1}}e_j^*)N_{n_{j-1}}^\perp
  \qquad\text{and}\qquad
  f_jx_{m_j}^* = N_{m_j - 1}(f_jx_{m_j}^*)N_{m_j -1}^\perp,
\]
each of the terms of the sums are in $\tn$,
and the ranges and cokernels of the terms are pairwise orthogonal, so that both
sums converge strongly. Now clearly for each $1\le i\le n$,
\(
  \|X_i y_n\|
    \le N^{-1/2}\sum_{j=1}^N \|X_i x_{m_j}\|
    < N^{1/2}/2n^2
    = 1/n
\) and, likewise $\|X^*_i y_n\| < 1/n$.
Further, $A_n X B_n = \sum_{j=2}^N x_{m_{j-1}}x_{m_j}^*$, so that
\[
  \|(I - A_n X B_n)y_n\| 
    = N^{-1/2}\Big\|\sum_{j=1}^N x_{m_j} - \sum_{j=2}^Nx_{m_{j-1}}\Big\|
    = N^{-1/2}
    < 1/n
\]
and $(A_n X B_n)^* = \sum_{j=2}^N x_{m_j}x_{m_{j-1}}^* = \sum_{j=1}^{N-1} x_{m_{j + 1}}x_{m_j}^*$,
so that
\[
  \|(I - A_n X B_n)^*y_n\| 
    = N^{-1/2}\Big\|\sum_{j=1}^N x_{m_j} - \sum_{j=1}^{N-1}x_{m_{j+1}}\Big\|
    = N^{-1/2}
    < 1/n.
\]
Note also that each of the $e_j, f_j, x_{m_j}$ for $1\le j \le N$ lie in the
range of $N_{m_N} - N_{m_1 - 1} \le N_{k_n} - N_{k_{n-1}}$. Thus
$A_n = (N_{k_n} - N_{k_{n-1}})A_n(N_{k_n} - N_{k_{n-1}})$,
$B_n = (N_{k_n} - N_{k_{n-1}})B_n(N_{k_n} - N_{k_{n-1}})$, and
$y_n = (N_{k_n} - N_{k_{n-1}})y_n$.

Having met all the requirements, the induction proceeds as stated, and we let
$A := \sum_{n=1}^\infty A_n$ and $B := \sum_{n=1}^\infty B_n$. Clearly
for any fixed $i\in\NN$,
\[
\max\{\|X_i y_n\|, \|X^*_i y_n\|\} < 1/n
\]
for all sufficiently large $n$ and so $X_iy_n, X^*_iy_n\rightarrow 0$.
Moreover since $A$ and $B$ are block diagonal with respect to $N_{n_k}$, as
is $X$, it follows that $(I-A X B)y_n=(I-A_n X B_n)y_n\rightarrow 0$
and $(I-A X B)^*y_n=(I-A_n X B_n)^*y_n\rightarrow 0$ and we are done.
\end{proof}

\begin{lemma}\label{chain-of-type-s-lemma}
Fix a nest $\N$ and a projection $N\in\N$ and
let $\S_i$ ($i\in\NN$) be a countable collection of countable sets of Type-S which
form a chain (i.e.\ for any $i,j$, either $S_i\subseteq S_j$ or $S_j\subseteq S_i$).
Then $\bigcup_{i\in\NN}\S_i$ is also of Type-S.
\end{lemma}

\begin{proof}
The proof is a routine countability argument.
Recall that the strong operator topology on $\N$ is metrizable; let $d$ be a metric for it.
Enumerate $\bigcup_{i\in\NN}\S_i$ and let the sets $\C_n$ ($n\in\NN$) consist of the
first $n$ terms of that enumeration.
Fix $n$ and suppose $N_m$ and $x_m$
have been chosen for $m < n$ so that $N_1<N_2<\cdots<N_{n-1} < N$,
$x_m = (N_m - N_{m-1})x_m$, and $\max\{\|X x_m\|, \|X^* x_m\|\} < 1/m$ for all $X\in \C_m$.
Each $X\in\C_n$ belongs to some $\S_i$ and since $\C_n$ is finite and $\{S_i\}$ is a chain,
$\C_n$ is contained in some $\S_i$. Therefore $\C_n$ is of Type-S. Using this fact,
we can find $N_{n-1} < N_n < N$ with $d(N_n, N) < 1/n$ and $x_n = (N_n - N_{n-1})x_n$
such that $\max\{\|X x_n\|, \|X^* x_n\|\} < 1/n$ for all $X\in \C_n$.
Continue this inductively to construct a strictly increasing sequence $N_n\rightarrow N$
and $x_n = (N_n - N_{n-1})$ for all $n\in\NN$ such that
$\max\{\|X x_n\|, \|X^* x_n\|\} < 1/n$ for all $X\in \C_n$
(taking $k_0=0$ to get the induction started). Each $X\in\bigcup_{i\in\NN}\S_i$
belongs to $\C_n$ for all sufficiently large $n$, and so the result follows with the
vectors so chosen.
\end{proof}

\begin{theorem}\label{diagonal-are-primitive-thm}
Assume the Continuum Hypothesis and let $\N$ be a nest. Then the diagonal ideals of $\tn$
are primitive ideals.
\end{theorem}

\begin{proof}
The result is trivial when the diagonal ideal is of type $\I_N^-$ with $N^- < N$ or
$\I_N^+$ with $N^+ > N$. For in either case the diagonal ideal is the kernel of
the representation $X\mapsto E X |_{E\H}$ where $E$ is an atom of $\N$ and whose range is
therefore all of $B(E\H)$, and so is irreducible. For the remainder of the proof, consider
only diagonal ideals which are not of this type.

Next, let $\I$ be a diagonal ideal of $\tn$ and suppose that $\I=\I^-_N$ for some $N = N^->0$ in $\N$.
It is enough to construct operators $A_X, B_X\in\tn$ for each operator $X\in\tn \setminus\I^-_N$,
such that the collection $\{I - A_X X B_X : X\in\tn\setminus\I^-_N \}$ generates a proper left
ideal of $\tn$. For then there is a maximal left ideal $\L$ which contains this family of operators
and the kernel of the left regular representation of $\tn$ on $\tn/\L$ is a primitive ideal, $\P$,
which by Remark~\ref{prim-ideals-remark} must exclude all $X\not\in\I^-_N$. Thus $\P\subseteq\I_N^-$.
Since every primitive ideal contains a diagonal ideal \cite[Theorem 4.9]{Ringrose:OnSoAlOp} and
the distinct diagonal ideals are incomparable \cite[Lemma 4.7]{Ringrose:OnSoAlOp},
it follows that $\I_N^- = \P$ and so is primitive.

Now consider the case when $\I = \I_N^+$ for some $N=N^+<I$ in $\N$. By the same reasoning,
it is enough to find $A_X, B_X\in\tn$ such that $\{I - A_X XB_X : P\in\tn\setminus\I_N^+\}$
is contained in a proper left ideal of $\tn$. To do this, we take adjoints and seek
$A_X, B_X \in \T(\N)^* = \T(\N^\perp)$ such that
$\{I - A_X X B_X : X\in\T(\N^\perp)\setminus\I_{N^\perp}^-\}$
is contained in a proper right ideal of $\T(\N^\perp)$. Since $\N$ is an arbitrary nest,
we can replace $\N^\perp$ with $\N$, to recast this as a second problem about $\I_N^-$ in $\tn$:
namely, to find $A_X, B_X\in\tn$ for each $X\in\tn\setminus\I^-_N$, such that
$\{I - A_X X B_X : X\in\tn\setminus\I^-_N \}$ generates a proper \emph{right}
ideal of $\tn$. We shall show in fact that the choice can be made so that the same set
of operators $\{I - A_X X B_X\}$ serves to generate both a proper left ideal and
a proper right ideal.
We shall construct these operators using transfinite recursion.

The cardinality of $\tn\setminus\I^-_N$ is equal to the cardinality of the contiuum
since every operator can be represented as a countable array of complex numbers.
Since we are assuming the Continuum Hypothesis, $\tn\setminus\I^-_N$ has cardinality $\aleph_1$
and so it can be put in bijective correspondence with the set of ordinals $a < \omega_1$
(where $\omega_1$ denotes the first uncountable ordinal).
Write this correspondence as $X_a$ ($a<\omega_1$).
To run the transfinite recursion, we suppose that for
some $a<\omega_1$ we have operators $A_b, B_b$ in $\tn$ for all $b<a$, and describe how to
obtain $A_a, B_a$. First, if the set $\{I - A_bX_bB_b : b < a \}$ is of \types\ then observe
that $\{ I - A_bX_bB_b : b < a \}$ is a countable collection and use
Lemma~\ref{countable-extension-lemma} to find $A_a, B_a \in \tn$ such that
$\{I - A_bX_bB_b : b \le a \}$ is also of \types.
On the other hand, if it happens that $\{I - A_bP_bB_b : b < a \}$ is not of \types\
then set $A_a = B_a = 0$. (This is a sink terminal state which we shall
prove momentarily is never in fact reached.)

Note that formally Lemma~\ref{countable-extension-lemma} assumes a countably infinite
collection of predecessors. However the case of finite $a$, or even $a=1$, can be covered
by padding the collection of predecessors with countably many repeated zeros. Note also
the recursion step involves an arbitrary choice of operators, which can easily be
resolved using the Axiom of Choice.

Having described a rule to construct $A_a, B_a$ with $(A_b, B_b)_{b < a}$ given, we apply the
principle of transfinite recursion to obtain $(A_a, B_a)_{a < \omega_1}$ where the transition
rule from the previous paragraph applies for every $a<\omega_1$. We next note that
for every $a < \omega_1$, $\S_a := \{I - A_bX_bB_b : b \le a \}$ is of \types.
For if this were not true, then we could find the least $a$ such that $\S_a$ is not
\types. Thus for each of the countably many $b<a$, $\S_b$ is countable and of \types\
and so by Lemma~\ref{chain-of-type-s-lemma}, $\bigcup_{b<a} \S_b$ is \types.
But $\bigcup_{b<a} \S_b = \{I - A_bX_bB_b : b < a \}$ and so
by the recursion step, $\S_a = \{I - A_bX_bB_b : b \le a \}$ is also \types.
Thus, by contradiction, each $\S_a$ is of \types\ and, in particular, generates a
proper left ideal of $\tn$ and a proper right ideal of $\tn$.
Now in general, the union of any chain of sets, each of which generates a proper 
left (resp.\ right) ideal, will also generate a proper left (resp.\ right) ideal.
Thus,
$\{I - A_aX_aB_a : a < \omega_1 \} = \bigcup_{a<\omega_1} \S_a$ generates a proper
left ideal and a proper right ideal, and the result follows.
\end{proof}

\begin{corollary}\label{diagonal-are-right-primitive-cor}
Assuming the Continuum Hypothesis, the diagonal ideals of $\tn$ are also right-primitive ideals,
that is to say, the annihilators of simple right modules.
\end{corollary}

\begin{proof}
The conjugate-linear anti-isomorphism $X\mapsto X^*$ maps $\tn$ to $\T(\N^\perp)$,
maps diagonal ideals to diagonal ideals, and converts left modules into right modules.
\end{proof}

We remark in passing that Theorem~\ref{diagonal-are-primitive-thm}
does provide a new proof of Ringrose's characterization of the Jacobson radical
of a nest algebra. For in view of Theorem~\ref{prim-contains-diagonal-thm}
\[
  \bigcap_{\I\text{ diagonal}} \I
    \subseteq \bigcap_{\P\text{ primitive}} \P,
\]
and the reverse inclusion follows from Theorem~\ref{diagonal-are-primitive-thm}.
Insofar as our result assumes the Continuum Hypothesis and also assumes
$\H$ is separable, this is, of course, substantially less general
than Ringrose's original proof.

\section{The left ideals of a nest algebra}\label{left-ideals-sect}

In this section we study the left ideals of nest algebras.
Definition~\ref{constructible-def} gives a method of specifying
left ideals and in Theorem~\ref{every-ideal_is-constructible-thm}
we shall see that every left ideal can be specified in this way.
We then introduce (Definition~\ref{strongly-constructive-def})
a stronger property which specifies many closed left ideals,
including the maximal left ideals. This leads to insights into the
structure of left ideals (Proposition~\ref{max-left-contains-projections-prp})
which we apply in the following sections.

\begin{definition}\label{constructible-def}
Let $\L$ be a left ideal of $\tn$. Say that $\L$ is \emph{constructible} if
there is a net indexed by a directed set $A$ consisting of pairs
$(N_\alpha, x_\alpha)$
of projections $N_\alpha \in \N$ and vectors $x_\alpha \in \H$ such that
\[
  \L = \{ X \in \tn : \lim_{\alpha \in A} \|(I - N_\alpha) X x_\alpha\| = 0 \}
\]
for every $X\in \tn$.
\end{definition}

\begin{lemma}\label{proper-or-not-lemma}
$\tn$ is itself a constructible ideal and, in general, the constructible ideal,
$\L$, specified by the net $(N_\alpha, x_\alpha)_{\alpha \in A}$ is equal to $\tn$
if and only if $\lim_\alpha N_\alpha^\perp x_\alpha = 0$.
\end{lemma}

\begin{proof}
If $N_\alpha^\perp x_\alpha \rightarrow 0$ then, for any fixed $X\in \tn$,
\(
  \| N_\alpha^\perp X x_\alpha \|
  = \| N_\alpha^\perp X N_\alpha^\perp x_\alpha \|
  \le \| X \|  \| N_\alpha^\perp x_\alpha \|
\)
and so $X \in \L$. Conversely, if $N_\alpha^\perp x_\alpha \not\rightarrow 0$, then $I \not\in \L$
and so $\L$ is proper.
\end{proof}

Note that if $(N_\alpha, x_\alpha)_{\alpha\in A}$ is a net in $\N\times\H$ and $X\in\tn$,
then $N_\alpha^\perp X x_\alpha = N_\alpha^\perp X N_\alpha^\perp x_\alpha$,
for all $\alpha$ and so without loss we can always assume that
$x_\alpha = N_\alpha^\perp x_\alpha$.

The following interpolation result of Katsoulis, Moore, and Trent
enables us to see that all left ideals are constructible.
In this context we remark that the results of \cite{KatsoulisMooreTrent:InNeAl}
have a precursor in Lance's \cite[Theorem 2.3]{Lance:SoPrNeAl}, introduced to study
the radical and diagonal ideals.

\begin{theorem}\label{kmt-thm}
\cite[Theorem 4]{KatsoulisMooreTrent:InNeAl}.
Let $X_1, \ldots, X_n$ and $Y$ be in $\tn$. Then there are $A_1, \ldots, A_n$
in $\tn$ such that
\[
  Y = \sum_{i=1}^n A_i X_i
\]
if and only if
\begin{equation}\label{kmt-criterion}
  \sup \left\{
    \frac{\|N^\perp Y x\|^2}{\sum_{i=1}^n \|N^\perp X_i x\|^2} : N \in \N, x \in \H 
  \right\} < \infty
\end{equation}
(where $0/0$ is interpreted as $0$).
\end{theorem}

\begin{theorem}\label{every-ideal_is-constructible-thm}
Every left ideal of a nest algebra is constructible.
\end{theorem}

\begin{proof}
Let $\L$ be a fixed left ideal of the nest algebra $\tn$ and take $A$ to be
the set of all 4-tuples $(F, \e, N, x)$ where $F$ is a finite subset of $\L$,
$\e >0$, $N \in \N$, and $x \in \H$, subject to the constraint that
$\| N^\perp X x \| < \e$ for all $X \in F$. 
This is a directed set if we say $(F, \e, N, x) \le (F', \e', N', x')$
when $F \subseteq F'$ and $\e \ge \e'$. For the relation is clearly reflexive and
transitive, and any pair of members of $A$, $(F, \e, N, x)$ and $(F', \e', N', x')$,
is dominated by $(F \cup F', \min\{\e, \e'\}, 0, 0)$. Define a net on $A$ with values in
$\N\times\H$ by the mapping which takes $\alpha := (F, \e, N, x) \in A$
to $(N_\alpha, x_\alpha)$ where $N_\alpha := N$, and $x_\alpha := x$.
We shall see that this net specifies $\L$ exactly.

On the one hand, trivially, if $X \in \L$ then for any $\e > 0$, the tuple
$\alpha_0 := (\{X\}, \e, 0, 0)$ belongs to $A$ and so for any $\alpha \ge \alpha_0$
$\| N_\alpha^\perp X x_\alpha \| < \e$. So, next, suppose on the other hand that
$Y \in \tn \setminus \L$.

Let an arbitrary $\alpha_0 := (\{X_1, \ldots, X_n\}, \e, M, x)$ in $A$ be given.
Since $Y \not\in \L$, there do not exist any $A_1, \ldots, A_n$ in $\tn$
such that $\sum_{i=1}^n A_i X_i = Y$.
Thus by Theorem~\ref{kmt-thm}, the supremum~(\ref{kmt-criterion}) is infinite, 
and so we can find $N \in \N$ and $y \in \H$ such that
\[
  \|N^\perp X_i y \| < \e \| N^\perp Y y \|
\]
for each $i = 1, \ldots, n$. Rescaling $y$, we obtain $N$ and $y$ such that
$\|N^\perp X_i y \| < \e$ and $\| N^\perp Y y \| = 1$. 
Thus $\beta := (\{X_1, \ldots, X_n\}, \e, N, y)$ is in $A$,
and we have $\beta \ge \alpha_0$ and $\|N_\beta^\perp Y x_\beta\| = 1$.
In other words, the net
$\| N_\alpha^\perp Y x_\alpha \|$ is frequently equal to 1, and so
$\| N_\alpha^\perp Y x_\alpha \| \not\rightarrow 0$.
\end{proof}

\begin{example}
The set $\F_\N$ of finite rank operators in $\tn$ is a two-sided ideal
of $\tn$ but is not norm-closed. We can specify this with the following
net. Let $A$ consist of the set of pairs $(F, x)$ where $F$ is a
finite-dimensional subspace of $\H$ and $x$ is a vector which is orthogonal
to $F$. For $\alpha=(F, x)\in A$ define $x_\alpha := x$ and $N_\alpha=0$.
Say $(F, x) \le (G, y)$ in $A$ if $F\subseteq G$. Clearly $T\in\tn$
belongs to $\F_\N$ if and only if there is a finite-dimensional space
$F$ such that $T$ vanishes on $F^\perp$. Since the vectors in the pairs
are unbounded, the condition $\| N_\alpha^\perp T x_\alpha\| < 1$ for
all $\alpha\ge(F, 0)$ is equivalent to $T$ vanishing on $F^\perp$.
\end{example}

\begin{example}
The set $\K_\N$ of compact operators in $\tn$ is a norm-closed two-sided
ideal of $\tn$. We can specify it with the following net, which is similar
to the previous example. Let $A$ consist of the set of pairs $(F, x)$ where $F$ is a
finite-dimensional subspace of $\H$ and $x$ is a unit vector which is orthogonal
to $F$. Again for $\alpha=(F, x)\in A$ define $x_\alpha := x$ and $N_\alpha=0$, and
say $(F, x) \le (G, y)$ in $A$ if $\F\subseteq G$. By the spectral theory,
an operator $T\in\tn$
belongs to $\K_\N$ if and only if for any $\e>0$ there is a finite-dimensional space
$F$ such that $\|T|_{F^\perp}\|<\e$, which is readily seen to be equivalent
to $\|N_\alpha^\perp T x_\alpha\|\rightarrow 0$.
\end{example}

The contrast between the last two examples, in which the net was unbounded in
one case and bounded in the other, motivates the following definition.

\begin{definition}\label{strongly-constructive-def}
Let $\L$ be a left ideal of $\tn$. Say that $\L$ is \emph{strongly constructible} if
it is constructible and a net $(N_\alpha, x_\alpha)_{\alpha \in A}$ specifying $\L$
can be found in which all the vectors $x_\alpha = N_\alpha^\perp x_\alpha$ have norm $1$.
\end{definition}

\begin{proposition}
Strongly constructible ideals are norm-closed.
\end{proposition}

\begin{proof}
Let $\L$ be strongly constructible and specified by $(N_\alpha, x_\alpha)_{\alpha \in A}$,
where $\| x_\alpha \| = 1$ for all $\alpha \in A$. Suppose the sequence of $X_n \in \L$
converges in norm to $X \in \tn$. Given $\e > 0$, find a fixed $n \in \NN$ such that
$\| X - X_n \| < \e/2$ and $\alpha_0 \in A$ such that
$\| N_\alpha^\perp X_n x_\alpha \| < \e/2$ for all $\alpha \ge \alpha_0$. Then
\[
  \| N_\alpha^\perp X x_\alpha \|
    \le \| X - X_n \|\|x_\alpha\| + \| N_\alpha^\perp X_n x_\alpha \|
    < \e
\]
\end{proof}

\begin{proposition}\label{maximal-strongly-constructible-prop}
The maximal left ideals of $\tn$ are strongly constructible.
\end{proposition}

\begin{proof}
Let $\L$ be a maximal left ideal which we suppose to be specified by the net
$(N_\alpha, x_\alpha)_{\alpha \in A}$. Without loss, assume that each
$x_\alpha = N_\alpha^\perp x_\alpha$. By Lemma~\ref{proper-or-not-lemma},
$x_\alpha \not\rightarrow 0$ and so there is an $\e_0>0$ such that
$\|x_\alpha\|$ is frequently at least $\e_0$. Let
$A' := \{ \alpha\in A : \|x_\alpha\| \ge \e_0 \}$ and
$x'_\alpha = x_\alpha / \|x_\alpha\|$ for $\alpha\in A'$.
Now $A'$ is a directed set and $(N_\alpha, x_\alpha)$ is a net on it.
Again by Lemma~\ref{proper-or-not-lemma},
the net $(N_\alpha, x'_\alpha)_{\alpha \in A'}$ specifies a proper ideal
which, furthermore, contains $\L$ since for $X\in\L$,
\[
  \| N_\alpha^\perp X x'_\alpha \| \le \frac{1}{\e_0} \; \| N_\alpha^\perp X x_\alpha \|
\]
for all $\alpha\in A'$ and the net on the right converges to zero since
$(N_\alpha, x_\alpha)_{\alpha \in A'}$ is a subnet of
$(N_\alpha, x_\alpha)_{\alpha \in A}$.
By maximality, the ideal which $(N_\alpha, x'_\alpha)_{\alpha \in A'}$ specifies
must equal $\L$.
\end{proof}\begin{proposition}\label{intersections-prp}
Arbitrary intersections of strongly constructible ideals are strongly constructible.
\end{proposition}

The proof is a consequence of the following simple result about nets.

\begin{lemma}\label{nets-lemma}
Fix a set $X$ and suppose that we have a family of nets in $X$ indexed by a set $K$,
which we denote by $(x^{(k)}_\alpha)_{a\in A_k}$. Then we can find a
net $(x_\alpha)_{\alpha \in A}$ in $X$ with the property that for any
$ E\subseteq X$, $(x_\alpha)_{\alpha \in A}$ is eventually in $E$ if and only if
for each $k \in K$, $(x^{(k)}_\alpha)_{\alpha \in A_k}$ is eventually in $E$.
\end{lemma}

\begin{proof}
Define $A$ to be set the set of pairs $(\sigma, k)$ where $\sigma$ is a section map
on the fibre bundle of $A_k$ over $K$ (i.e., for each $k \in K$, $\sigma(k) \in A_k$),
and $k$ is an arbitrary member of $K$. Put a relation on $A$ by declaring
$(\sigma, k) \le (\tau, l)$ if $\sigma(i) \le_i \tau(i)$ for all $i\in K$ (the relation $\le_i$
is the directed relation defined on $A_i$). This is a symmetric and transitive relation.
Moreover, if $(\sigma, k)$ and $(\tau, l)$ are in $A$ then for each $i\in K$ we can
find an element of $A_i$ which dominates both $\sigma(i)$ and $\tau(i)$. By the Axiom of
Choice there is therefore a section map $\rho$ such that $\rho(i)$ dominates both $\sigma(i)$ and
$\tau(i)$ for all $i \in K$. Taking an arbitrary $i \in K$, then $(\rho, i)$ dominates both
$(\sigma, k)$ and $(\tau, l)$ in $A$. Thus $A$ is a directed set, and we define the net
$(x_{(\sigma, k)})_{(\sigma, k) \in A}$ by $x_{(\sigma, k)} :=x^{(k)}_{\sigma(k)}$.

Now, on one hand, suppose that $(x_{(\sigma, k)})$ is eventually in $E \subseteq X$.
Thus there is a $(\sigma_0, k_0) \in A$ such that
$x_{(\sigma, k)} \in E$ for all $(\sigma, k) \ge (\sigma_0, k_0)$.
Fix $k\in K$ and consider $\alpha_0 := \sigma_0(k) \in A_k$. If $\alpha \ge_k \alpha_0$
then define
$\sigma(i) := \sigma_0(i)$ for all $i \not= k$ and $\sigma(k) := \alpha$.
Then $(\sigma, k) \ge (\sigma_0, k_0)$ and so $x^{(k)}_\alpha = x_{(\sigma, k)} \in E$.
This shows that for each $k \in K$, $(x^{(k)}_\alpha)_{\alpha \in A_k}$ is eventually in
$E$.

Conversely, let $E \subseteq X$ and suppose that for every $k \in K$,
$(x^{(k)}_\alpha)_{\alpha \in A_k}$ is eventually in $E$. That is to say,
for each $k \in K$, we can can find an $\alpha_0 \in A_k$ such that
$x^{(k)}_\alpha \in E$ for all $\alpha \ge_k \alpha_0$ in $A_k$. Again by the
Axiom of Choice we pick one such $\alpha_0$ for each $k\in K$ and obtain a section
$\sigma_0$ such that for each $k \in K$ and
$\alpha \ge_k \sigma_0(k)$ in $A_k$, we have $x^{(k)}_\alpha \in E$. Pick an arbitrary
$k_0 \in K$ and then suppose $(\sigma, k) \ge (\sigma_0, k_0)$. This means that,
in particular, $\sigma(k) \ge_k \sigma_0(k)$, so that
$x_{(\sigma, k)} = x^{(k)}_{\sigma(k)} \in E$. We conclude that the net
$(x_{(\sigma, k)})_{(\sigma, k) \in A}$ is eventually in $E$.
\end{proof}

The proof of Proposition~\ref{intersections-prp} now follows straightforwardly.

\begin{proof}[Proof (of Proposition~\ref{intersections-prp})]
Let $\L_k$ ($k \in K$) be a collection of strongly constructible left ideals.
Writing $\H_1$ for the set of unit vectors in $\H$, for each
$k \in K$ there are directed sets $A_k$ and nets
$(N^{(k)}_\alpha, x^{(k)}_\alpha) \in \N \times \H_1$ for $\alpha\in A_k$
such that an $X\in\tn$ belongs to $\L_k$ if and
only if $\lim_{\alpha\in A_k}\|(I - N^{(k)}_\alpha) X x^{(k)}_\alpha\| = 0$.

By Lemma~\ref{nets-lemma}, find a new net $(N_\alpha, x_\alpha)_{\alpha \in A}$
in $\N \times \H_1$ which is eventually in a subset of
$\N \times \H_1$ if and only if
each of the $(N^{(k)}_\alpha, x^{(k)}_\alpha)_{\alpha\in A_k}$ are eventually in
that set. Fix $X \in \tn$ and let $\e > 0$ be given. Let
\[
  E_\e : = \{ (N, x) \in \N \times \H_1 : \| (I - N) X x \| < \e \}
\]
Clearly $X \in \bigcap_{k \in K} \L_k$ iff
for every $k\in K$ and every $\e > 0$,
$(N^{(k)}_\alpha, x^{(k)}_\alpha)_{\alpha\in A_k}$ is eventually in $E_\e$.
This happens iff for every $\e > 0$, $(N_\alpha, x_\alpha)_{\alpha \in A}$
is eventually in $E_\e$, which in turn happens iff
$\lim_{\alpha \in A} \|(I - N_\alpha) X x_\alpha\| = 0$. Thus
$\bigcap_{k \in K} \L_k$ is strongly constructible.
\end{proof}

\begin{corollary}
Every proper left ideal $\L$ of $\tn$ is contained in a smallest strongly
constructible left ideal, which we shall call the
\emph{strongly constructible hull} of $\L$.
\end{corollary}

\begin{corollary}\label{primitives-are-strongly-constructible-cor}
The primitive ideals of $\tn$ are strongly constructible.
\end{corollary}

\begin{proof}
Every primitive ideal is the intersection of the maximal left ideals which contain it
\cite[\sectionmark{24}, Proposition 12 (iv)]{BonsallDuncan:CoNoAl}.
The result follows by By Propositiona~\ref{maximal-strongly-constructible-prop}
and~\ref{intersections-prp}.
\end{proof}

\begin{example}
In particular, the maximal two-sided ideals of $\tn$, being primitive, are strongly
constructible. Recall that the \emph{strong radical} of a unital algebra is the intersection of
all its maximal two-sided ideals. In \cite[Theorem 3.2]{Orr:MaIdNeAl} we saw that if $\tn$
is a continuous nest algebra then any norm-closed, two-sided ideal of $\tn$ which contains
the strong radical is the intersection of the maximal two-sided ideals which contain it.
Thus by Propositions~\ref{intersections-prp} and~\ref{primitives-are-strongly-constructible-cor},
all such ideals are strongly constructible.
\end{example}

\begin{corollary}
All norm-closed, two-sided ideals of a continuous nest algebra which contain the
strong radical are strongly constructible.
\end{corollary}

\begin{question}
Is every norm-closed left ideal of a nest algebra strongly constructible?
\end{question}

Strongly constructible ideals are also characterized by two ostensibly
weaker conditions:

\begin{proposition}\label{equivalent-strongly-constructible-criteria-prop}
Let $\L$ be a proper left ideal of $\tn$. The following are equivalent:
\begin{enumerate}
  \item $\L$ is strongly constructible.
  \item $\L$ can be specified by a net $(N_\alpha, x_\alpha)_{\alpha\in A}$
      where $\|x_\alpha\| \le 1$ for all $\alpha\in A$.
  \item $\L$ can be specified by a net $(N_\alpha, x_\alpha)_{\alpha\in A}$
      where $\|x_\alpha\|$ is bounded.
\end{enumerate}
\end{proposition}

\begin{proof}
Clearly $(1) \Rightarrow (2) \Rightarrow (3)$ and so it remains to prove $(3) \Rightarrow (1)$.
Suppose $(N_\alpha, x_\alpha)_{\alpha\in A}$ specifies $\L$ and $\|x_\alpha\|$ is bounded.
Since $\L$ is proper, by Lemma~\ref{proper-or-not-lemma} $x_\alpha\not\rightarrow 0$
and so there is an $\e_0$ such that $\|x_\alpha\|$ is frequently at least $\e_0$. For each
$k\in\NN$ set
\[
  A_k := \{ \alpha \in A : \|x_\alpha\| \ge \e_0 / k \}
\]
Each $A_k$ is a directed set (with the order relation inherited from $A$) and the
restricted net $(N_\alpha, x_\alpha)_{\alpha\in A_k}$ defines a left ideal $\L_k$.
Since the $x_\alpha$ are bounded away from zero on $A_k$, we can normalize and see
each $\L_k$ is strongly constructible. It remains to check that
$\L = \bigcap_{k\in\N} \L_k$ and then the result will follow by
Proposition~\ref{intersections-prp}.

Clearly since each $A_k \subseteq A$, also $\L \subseteq \L_k$ and so
$\L \subseteq \bigcap_{k\in\N} \L_k$. Suppose $X\not\in\L$. Then
$(I-N_\alpha)Xx_\alpha \not\rightarrow 0$ and so there is an $\e_1>0$ such that
$\|(I-N_\alpha)Xx_\alpha\| \ge \e_1$ frequently. Choose $k > \e_0\|X\| / \e_1$
so that then whenever $\|(I-N_\alpha)Xx_\alpha\| \ge \e_1$ then
\[
  \|X\|\|x_\alpha\| \ge \|(I-N_\alpha)Xx_\alpha\| \ge \e_1 > \|X\| \frac{\e_0}{k}
\]
and thus $\alpha\in A_k$. It follows that $\|(I-N_\alpha)Xx_\alpha\| \ge \e_1$
frequently on $A_k$, and so $X\not\in\L_k$.
\end{proof}

\begin{proposition}\label{max-left-contains-projections-prp}
Let $\L$ be a maximal left ideal in $\tn$ and let $P_n$ be a sequence of pairwise
orthogonal projections in $\L$. There is a subsequence $P_{n_k}$ such that the
projection $\sum_{n=1}^\infty P_{n_k}$ belongs to $\L$.
\end{proposition}

\begin{proof}
By Proposition~\ref{maximal-strongly-constructible-prop}, $\L$ is strongly constructible,
say by a net $(N_\alpha, x_\alpha)$ where each $x_\alpha$ is a unit vector in the range of
$N_\alpha$. By Kelley's Theorem, this net has a universal subnet, which specifies a proper
ideal containing $\L$, hence in fact specifies $\L$ itself. Thus we may assume
$(N_\alpha, x_\alpha)$ is universal.

The proof now proceeds by means of a fairly routine diagonal argument.
For any $S\subseteq\NN$ write $P(S) : = \sum_{n\in S} P_n$.
Take $S_0:=\NN$ and split $S_0$ into two infinite sets, $S_0'$ and $S_0''$.
If $\|P(S_0')x_\alpha\|$ and $\|P(S_0'')x_\alpha\|$ are each eventually
greater than $1/\sqrt{2}$ then
$\|P(S_0)x_\alpha\|^2 = \|P(S_0')x_\alpha\|^2 + \|P(S_0'')x_\alpha\|^2$
is eventually greater than $1$, which is impossible. Since $(N_\alpha, x_\alpha)$
is universal that means at least one of $\|P(S_0')x_\alpha\|$,
$\|P(S_0'')x_\alpha\|$ is eventually no greater than $1/\sqrt{2}$; without loss
suppose that $\|P(S_0')x_\alpha\| \le 1/\sqrt{2}$ eventually, and set $S_1 := S_0'$.

Now decompose $S_1 = S_1' \cup S_1''$ in the same way as the union of infinite subsets
and, as before, we conclude that at least one of $\|P(S_1')x_\alpha\|$, $\|P(S_1'')x_\alpha\|$
is eventually no greater than $(1/\sqrt{2})^2$. Take $S_2$ to be one of $S'_1$, $S''_1$
for which this holds. Proceeding in this way we obtain a sequence
$S_0\supseteq S_1\supseteq S_2\supseteq \ldots$
of infinite subsets of $\NN$ such that for each $k$, eventually
$\|P(S_k) x_\alpha\| \le (1/\sqrt{2})^k$. Now take $n_k$ to be the $k$th element
of $S_k$ in order, which is a strictly increasing sequence,
and let $S := \{n_k\}$. Thus $S \setminus S_k$ is finite for all $k$.

Finally, write $P := P(S)$ and, given $\e>0$, take $k$ such that $(1/\sqrt{2})^k < \e$.
For all sufficiently large $\alpha$
\begin{eqnarray*}
  \|Px_\alpha\| &=& \|(P P(S_k) + P P(S_k)^\perp)x_\alpha\|                  \\
                &\le& \|P(S_k)x_\alpha\| + \|P(S\setminus S_k)x_\alpha\|     \\
                &\le& \e + \sum_{n\in S\setminus S_k} \| P_n x_\alpha \|
\end{eqnarray*}
But the sum in the last line is finite and so is eventually less than $\e$.
We can conclude $\|N_\alpha^\perp P x_\alpha\| = \|Px_\alpha\| \rightarrow 0$,
so that $P\in\L$.
\end{proof}

\begin{corollary}\label{max-right-contains-projections-cor}
Let $\J$ be a maximal right ideal in $\tn$ and let $P_n$ be a sequence of pairwise
orthogonal projections in $\R$. There is a subsequence $P_{k_n}$ such that the
projection $\sum_{k=1}^\infty P_{k_n}$ belongs to $\R$.
\end{corollary}

\begin{proof}
The result follows on taking adjoints and working in $\T(\N^\perp)$.
\end{proof}

\section{Atomic nest algebras}\label{atomic-nest-algebras-sect}

In this section we shall focus on atomic nest algebras and relate the character of
primitive ideals to the family of diagonal operators they contain. Observe
that if $\P$ is a primitive ideal of $\tn$ then $\P\cap\dn$ is a norm-closed
two-sided ideal of the C$^*$-star algebra $\dn$ and is therefore a $*$-ideal.
In many interesting cases the nest is multiplicity-free so that $\dn$ is an
abelian C$^*$-algebra.

\begin{proposition}\label{prim-max-prp}
Let $\N$ be an atomic nest and $\J$ a two-sided ideal in $\tn$.
Then $\J$ is a maximal two-sided ideal if and only if $\J\cap\dn$
is a maximal two-sided ideal of $\dn$.
\end{proposition}

\begin{proof}
Suppose $\J$ is maximal. Then by \cite[Theorem 3.8]{Orr:MaTwSoIdNeAl}, $\J$
contains $\R^\infty_\N$. It follows that $\J = (\J\cap\dn) \oplus \R^\infty_\N$.
If $\J\cap\dn$ is not maximal then there is a larger proper ideal $\D_0$ of $\dn$.
But then $\D_0\oplus\R^\infty_\N$ is a proper ideal of $\tn$ and strictly larger
than $\J$, contrary to fact.

Suppose on the other hand that $\J\cap\dn$ is maximal. By
\cite[Theorem 10.2]{DavidsonOrr:PrBiNeAl} $\R^\infty_\N$ is generated
as a two-sided ideal by a generator which is the sum of three commutators
$[G_i, P_i]$ ($i=1,2,3$) where $G_i\in\tn$ and $P_i$ is a projection in the core
$\C(\N)$ of $\tn$.
(Recall that the \emph{core} of a nest algebra is the abelian von~Neumann algebra
generated by $\N$.)
Now since $\J\cap\dn$ is a maximal ideal of $\dn$, and
the $P_i$ are in the centre of $\dn$, it follows that one of $P_i, P_i^\perp$
must lie in $\J\cap\dn$ for each $i$. Thus in any event the commutators
$[G_i, P_i] = [G_i, P^\perp_i]$ belong to $\J$ and so $\J$
contains $\R^\infty_\N$. Thus, again, $\J = (\J\cap\dn) \oplus \R^\infty_\N$.
If $\J$ is not maximal then there is a larger proper ideal $\J_0$ of $\tn$.
But then since $\J_0$ also contains $\R^\infty_\N$,
$\J_0 = (\J_0\cap\dn) \oplus \R^\infty_\N$ and so $\J_0\cap\dn$ is a proper
ideal of $\dn$ and larger than $\J\cap\dn$, contrary to fact.
\end{proof}

The proof of Proposition~\ref{prim-max-prp} is deceptively straightforward.
In fact the result cited from \cite{Orr:MaTwSoIdNeAl} depends on
Marcus, Spielman, and Srivastava's proof~\cite{MarcusSpielmanSrivastava:MiChPo}
of the Paving Theorem. Recall (Definition~\ref{isubp-def}) that we write
$\I_\P$ for the unique diagonal ideal contained by the primitive ideal $\P$.

\begin{proposition}\label{prim-contains-cpts-prop}
Let $\N$ be an atomic nest, let $\P$ be a primitive ideal of $\tn$,
and suppose $\P\not=\I_\P$. Then there are non-zero projections in
$\P\setminus\I_\P$.
\end{proposition}

\begin{proof}
We shall prove the result in the case when $\I_\P = \I^-_N$ for some $N>0$ in $\N$.
If, instead, $\I_\P = \I_N^+$ for some $N<I$ then we take adjoints and apply the result
to $\I^-_{N^\perp}\subsetneq\P^*\subseteq\tn^* = \T(\N^\perp)$.
In this case $\P$ is a \emph{right} primitive ideal of $\T(\N^\perp)$ and so we shall
take care that our proof accommodates the case when $\P$ is either left or right primitive.

If $\P$ is a left primitive ideal, let $\J$ be a maximal left ideal
such that $\P$ is the kernel of the left regular module action of $\tn$ on $\tn / \J$.
In the case that $\P$ is right primitive, let $\J$ be a maximal right ideal such that
$\P$ is the kernel of the right regular module action of $\tn$ on $\tn / \J$.

Suppose that $N^- < N$. Note that $\rk(N - N^-)$ cannot be finite for if it were
then $\I_\P = \I^-_N$ would be a maximal ideal of $\tn$ and so $\P=\I_\P$, contrary to hypothesis. If $\rk({N-N^-}) = \infty$ then the only proper ideal strictly containing
$\I_N^-$ is $\{X\in\tn : (N-N^-)X(N-N^-) \text{ is compact}\}$, which must therefore
equal $\P$. Any finite rank projection of the form $P=(N-N^-)P(N-N^-)$ will
serve to establish the result in this case.

For the remainder of the proof, assume that $N=N^-$ and take $X\in\P\setminus\I_N^-$.
By Lemma~\ref{block-diagonal-lemma}
there are $A,B\in\tn$ such that $AXB$ is block diagonal with respect to some
sequence $M_k$ of nest projections strictly increasing to $N$ and each of the
blocks has norm greater than $1$. Replacing $X$ with $AXB$ we can assume
$X = \sum_{k=1}^\infty(M_k - M_{k-1})X(M_k - M_{k-1})$ where the norm of each
term is greater than $1$.

Consider the sequence of intervals $M_{2k+1} - M_{2k}$. These are each
in $\I_N^-$ and so in $\J$.
By Proposition~\ref{max-left-contains-projections-prp} and
Corollary~\ref{max-right-contains-projections-cor}, whether $\J$ is assumed
to be maximal right or maximal left, there is a subsequence $k_n$
such that $\J$ contains $\sum_{n=1}^\infty M_{2k_n + 1} - M_{2k_n}$.
Then for each $n$ find an atom $N^+_n-N_n \le M_{2k_n + 1} - M_{2k_n}$.
Choose vectors $e_n, f_n, g_n$ such that
$e_n e_n^* \le N^+_n-N_n$ and
$f_n$ and $g_n$ are in the range of $M_{2k_n + 2} - M_{2k_n + 1}$ with
$\|f_n\| > \|g_n\| = 1$ and $f_n = X g_n$. Thus,
\[
  V := \sum_{n=1}^\infty e_n e_{n+1}^* =
      \left(\sum_{n=1}^\infty \|f_n\|^{-1} e_n f_n^*\right)
          X\left(\sum_{n=1}^\infty \|f_n\|^{-1}g_n e_{n+1}^*\right)
\]
where both of the sums converge strongly and are in $\tn$ because
\[
  e_nf_n^* = M_{2k_n + 1}(e_nf_n^*) M^\perp_{2k_n + 1}
\]
and
\[
  g_ne_{n+1}^* = M_{2k_n + 2}(g_ne_{n+1}^*)M_{2k_{n+1}}^\perp
               = M_{2k_n + 2}(g_ne_{n+1}^*)M_{2k_n + 2}^\perp
\]
since $k_{n+1}\ge k_n + 1$.

 Thus $V\in\P$. Let
$P := \sum_{k=1}^\infty e_{2k} e_{2k}^* \le \sum_{k=1}^\infty N^+_{2k} - N_{2k}$,
which is dominated by a projection in $\J$ and so is also in $\J$.
We shall show that $P\in\P$.

Suppose for a contradiction that $P\not\in\P$.
It follows, as observed in Remark~\ref{prim-ideals-remark},
that there are $A, B\in\tn$ such that
$I - APB \in\J$. We can assume that $A=AP$ and $B=PB$.
Write $A=A_1+A_2$ where
\[
  A_1 := \sum_{k=1}^\infty N_{2k-1}^\perp A (N^+_{2k} - N_{2k})
  \quad\text{and}\quad
  A_2 := A-A_1
\]
so that $A_2(N^+_{2k} - N_{2k}) = N_{2k-1} A (N^+_{2k} - N_{2k})$.
Likewise, write $B=B_1+B_2$ where
\[
  B_1 := \sum_{k=1}^\infty (N^+_{2k} - N_{2k}) B N_{2k+1}
  \quad\text{and}\quad
  B_2 := B-B_1
\]
so that $(N^+_{2k} - N_{2k})B_2 = (N^+_{2k} - N_{2k}) B N_{2k+1}^\perp$.
The sums for $A_1$ and $B_1$ converge strongly because the sequences of terms
are norm-bounded and have pairwise orthogonal ranges and cokernels.

Now set $A_2' := A_2 V^*$
and $B_2' := V^* B_2$. From the following computations we see that $A_2'$ and $B_2'$
are in $\tn$ since the terms of the sums are in $\tn$:
\begin{align*}
  A_2' &= A_2PV^* = \sum_{k=1}^\infty A_2 (N^+_{2k} - N_{2k}) V^*
      = \sum_{k=1}^\infty N_{2k-1} A (N^+_{2k} - N_{2k}) V^*N_{2k-1}^\perp \\
  B_2' &= V^*PB_2 = \sum_{k=1}^\infty V^* (N^+_{2k} - N_{2k}) B_2
      = \sum_{k=1}^\infty N^+_{2k+1}V^* (N^+_{2k} - N_{2k}) B N_{2k+1}^\perp
\end{align*}

Furthermore, since $VV^* = \sum_{k=1}^\infty e_{k}e_{k}^*$ and
$V^*V = \sum_{k=1}^\infty e_{k+1}e_{k+1}^* = \sum_{k=2}^\infty e_{k}e_{k}^*$,
we have that
\[
  A_2 = A_2 P = A_2PV^*V = A'_2 V \in \P
\]
and
\[
  B_2 = PB_2 = VV^*PB_2 = VB'_2 \in  \P
\]
Since $I - (A_1+A_2)P(B_1+B_2)\in \J$, it now follows that also $I - A_1PB_1\in\J$.

Now note that
\begin{align*}
  A_1 P B_1
      & = \sum_{k=1}^\infty A_1 (N^+_{2k} - N_{2k}) B_1 \\
      & = \sum_{k=1}^\infty N_{2k-1}^\perp A (N^+_{2k} - N_{2k}) B N_{2k+1} \\
      & = \sum_{k=1}^\infty (N_{2k}^+ - N_{2k-1}) A (N^+_{2k} - N_{2k}) B (N_{2k+1} - N_{2k}) \\
      & = \sum_{k=1}^\infty (N_{2k}^+ - N_{2k-1}) C_k (N_{2k+1} - N_{2k})
\end{align*}
where $C_k := A (N^+_{2k} - N_{2k}) B$.
We can decompose $A_1PB_1$ in two ways, either as
\[
  \sum_{k=1}^\infty (N_{2k}^+ - N_{2k}) C_k (N_{2k+1} - N_{2k})
      + \sum_{k=1}^\infty (N_{2k} - N_{2k-1}) C_k (N_{2k+1} - N_{2k})
\]
or as
\[
  \sum_{k=1}^\infty (N_{2k}^+ - N_{2k-1}) C_k (N_{2k}^+ - N_{2k})
      + \sum_{k=1}^\infty (N_{2k}^+ - N_{2k-1}) C_k (N_{2k+1} - N_{2k}^+)
\]
These two cases are of the form $PY + Z$ and $YP+Z$ respectively where in both cases $Z$
is nilpotent.
Recall that $P\in\J$ and so, whether $\J$ is a maximal left ideal or a maximal right ideal,
we conclude that $I - Z \in \J$, which is impossible since this is invertible and $\J$
is proper. From this contradiction we conclude that $P\in\P$.
\end{proof}

\begin{theorem}\label{diag-prim-cases-thm}
Let $\N$ be an atomic nest and let $\P$ be a primitive ideal of $\tn$.
\begin{enumerate}
\item\label{diag-prim-cases-thm:a}
If $\P\cap\dn$ is a maximal two-sided ideal of $\dn$ then $\P$ is
a maximal two-sided ideal of $\tn$.

\item\label{diag-prim-cases-thm:b}
If $\P\cap\dn$ is equal to $\I\cap\dn$ for some diagonal ideal $\I$
then $\P$ is a diagonal ideal and, in fact, $\P=\I$.
\end{enumerate}
\end{theorem}

\begin{proof}
Case~(\ref{diag-prim-cases-thm:a}) is just Proposition~\ref{prim-max-prp}. To prove
Case~(\ref{diag-prim-cases-thm:b}), suppose that $\P\cap\dn = \I\cap\dn$ for some
diagonal ideal $\I$. First observe that $\I_\P\cap\dn \subseteq\P\cap\dn=\I\cap\dn$.
Now, distinct diagonal ideals contain complementary projections (see the proof
of \cite[Lemma 4.8]{Ringrose:OnSoAlOp} for this fact) and so $\I$ must equal $\I_\P$.
But now if $\P \not=\I_\P$ then by Proposition~\ref{prim-contains-cpts-prop},
$\P$ contains projections which are not in $\I_\P$, contrary to hypothesis.
\end{proof}

We can now distinguish three classes of primitive ideals based on the
diagonal operators they contain.
The first class ($\pmax$) consists of primitive ideals for which $\P\cap\dn$
is a maximal ideal of $\dn$, and this consists of the maximal two-sided ideals of $\tn$.
The second class ($\pmin$) consists of primitive ideals for which $\P\cap\dn = \I\cap\dn$
for some diagonal ideal $\I$ and this class consists of diagonal ideals.
The third class ($\pint$) consists of the remaining primitive ideals for which $\P\cap\dn$
takes neither its minimal nor its maximal values.

The maximal ideals of a general nest algebra were completely described in
\cite[Corollary 3.10]{Orr:MaTwSoIdNeAl}. In particular when $\N$ is atomic
the ideals in $\pmax$ are precisely the ideals of the form $\D_0 \oplus\R^\infty_\N$
where $\D_0$ is a maximal two-sided ideal of $\dn$. The ideals in $\pmin$ are
the primitive ideals which are also diagonal ideals. Trivially all ideals
of the form $\I_N^-$ where $N>N^-$ (or, equivalently, $I_N^+$ where $N<N^+$)
are included in this class. (See the first paragraph of the proof of
Theorem~\ref{diagonal-are-primitive-thm} for details.) By
Theorem~\ref{diagonal-are-primitive-thm}, if we assume the Continuum Hypothesis
then $\pmin$ consists of \emph{all} the diagonal ideals. Without the assumption
of the Continuum Hypothesis we cannot say which additional diagonal ideals belong
to $\pmin$. The structure of $\pint$ is more delicate. In the following section we will
see examples of representatives of all three classes.

\section{The infinite upper triangular operators}\label{infinite-upper-sect}

Throughout this section, let $\H = \ell^2(\NN)$ and consider the algebra $\tnn$ of all
upper triangular operators with respect to the standard basis of $\ell^2(\NN)$.
Recall that we write $\{e_i\}_{i=1}^\infty$ for the standard basis and let $N_n$
be the projection onto the span of $\{e_1,\ldots,e_n\}$, and
$\N := \{N_n : n\in\NN\} \cup \{0, I\}$. Then $\tnn := \tn$ is the algebra
of infinite upper triangular operators with respect to the $e_i$ and $\R^\infty_\N$
is simply the ideal of infinite \emph{strictly} upper triangular operators.
Moreover, the diagonal ideals of $\tnn$ are precisely the ideals
$\I_1, \I_2, \I_3,\ldots; \I_\infty$ where $\I_n := \I_{N_n}^-$
for $1\le n <\infty$ and $\I_\infty := \I_I^-$. Note that $\I_\infty$ coincides with
the compact operators of $\tnn$, a fact which we shall develop below.

\subsection{The quasitriangular algebra}

Let $\K(\H)$ be the set of all compact operators in $B(\H)$ and write $\qt(\NN)$ for the
quasitriangular algebra $\T(\NN) + \K(\H)$. By \cite{FallArvesonMuhly:PeNeAl} and,
in more generality, \cite{DavidsonPower:BeApCAl}, $\qt(\NN)$ is a norm-closed algebra
in $\bh$ and the canonical isomorphism between $\qt(\NN)/\K(\H)$ and
$\T(\NN)/(\T(\NN)\cap\K(\H))$ is isometric.

\begin{corollary}
Assuming the Continuum Hypothesis,
$\T(\NN)/(\T(\NN)\cap\K(\H))$ is a left (resp.\ right) primitive algebra.
\end{corollary}

\begin{proof}
$\K(\H) \cap \tnn = \I_\infty$, which is a left primitive ideal by
Theorem~\ref{diagonal-are-primitive-thm} and a right primitive ideal by
Corollary~\ref{diagonal-are-right-primitive-cor}.
\end{proof}

\begin{corollary}
Assuming the Continuum Hypothesis, $\qt(\NN)/\K(\H)$
is a left (resp.\ right) primitive algebra, and $\K(\H)$
is a  left (resp.\ right) primitive ideal in $\qt(\NN)$.
\end{corollary}

\subsection{A catalogue of primitive ideals}

Clearly $\pmin$ contains $\{\I_1, \I_2, \ldots\}$. Assuming the Continuum Hypothesis
then by Theorem~\ref{diagonal-are-primitive-thm}
\[
  \pmin = \{\I_1, \I_2, \ldots\} \cup \{\I_\infty\}
\]
By \cite[Corollary 3.10]{Orr:MaTwSoIdNeAl} the ideals of $\pmax$ are precisely
the ideals of the form $\D_0 \oplus\R^\infty_\N$ where $\D_0$ is a maximal ideal
of $\dn$. In this case $\dn$ is naturally identified with $\ell^\infty(\NN)$
and its maximal ideal space with the sequences vanishing at points of $C(\beta\NN)$.
The maximal ideals of $\tnn$ corresponding to points of $\NN$ are precisely the
$\I_n$ and so we can write
\[
  \pmax = \{\I_1, \I_2, \ldots\} \cup \{\D_x \oplus\R^\infty_\N : x\in\beta\NN\setminus\NN\}
\]
where $\D_x$ is the maximal ideal of $\dn$ corresponding to sequences in
$\ell^\infty(\NN)$ vanishing at $x\in\beta\NN$.

There remains the set $\pint$ of primitive ideals which are neither diagonal ideals
nor maximal ideals.
These are the primitive ideals $\P$ where $\P\cap\dn$ is a closed ideal of $\dn$
corresponding to an ideal of $\ell^\infty(\NN)$ which strictly contains $c_0(\NN)$
and is not maximal. We cannot give a complete catalogue of these ideals but
we can provide a rich set of examples. 

Consider the following special case of a general construction of epimorphisms
between nest algebras, taken from Corollary~5.3 of \cite{DavidsonHarrisonOrr:EpNeAl}.
Let $0\le m_k < n_k < +\infty$ be integers such that the intervals
$(m_k, n_k]$ are pairwise disjoint and let $\U$ be a free ultrafilter
on $\NN$. Suppose that $\lim_{k\in\U} n_k - m_k = +\infty$.
Let $U_k:\ell^2(\NN)\rightarrow\ell^2(\NN)$
be the partial isometry mapping $e_i$ to $e_{i - m_k}$ when
$m_k < i \le n_k$ and zero otherwise. For $X\in\tnn$ define
\[
  \phi(X) := \lim_{k\in\U} U_k X U_k^*
\]
where convergence is in the weak operator topology and the limit always exists
by \textsc{WOT}-compactness of the unit ball.
Then by \cite[Corollary~5.3]{DavidsonHarrisonOrr:EpNeAl} this map is an epimorphism
of $\tnn$ onto $\tnn$.
Note also that $\phi$ is a *-homomorphism of the diagonal of $\tnn$ onto itself.

If $\phi$ is such an epimorphism of $\tnn$ onto $\tnn$ and $\pi$ is an irreducible
representation of $\tnn$ then clearly $\pi\circ\phi$ is also an irreducible
representation of $\tnn$.
If $\ker{\pi}$ is in $\pmax$ then so is $\ker \pi\circ\phi$. However, as we shall see,
if $\ker\pi \in\pmin\setminus\pmax$ then $\ker \pi\circ\phi$ will be
in $\pint$ and this provides a rich supply of examples of
primitive ideals in $\pint$.

Assuming the Continuum Hypothesis, $\I_\infty\in\pmin\setminus\pmax$, so consider
the primitive ideal $\P = \phi^{-1}(\I_\infty)$.
Note that $\phi$ annihilates $\I_\infty$ and so $\I_\infty$ is the unique diagonal ideal
in $\P$.
Writing $\Delta(X)$ for the diagonal
expectation $\sum_{k=1}^\infty (N_{n_k} - N_{m_k})X(N_{n_k} - N_{m_k})$,
observe that $\ker\Delta\subseteq\ker\phi\subseteq\P$ and so $\P\not=\I_\infty$.
Thus $\P\not\in\pmin$. On the other hand, $\P\not\in\pmax$ since, by
\cite[Theorem 3.8]{Orr:MaTwSoIdNeAl}, every maximal ideal of $\tnn$ contains
$\R^\infty_\N$, but $\P$ does not contain the unilateral backward shift $U$
since $\phi(U) = U \not\in \I_\infty$. Thus $\P\not\in\pmin$ and $\P\not\in\pmax$,
and so $\P\in\pint$.

In fact this construction readily yields uncountably many incomparable ideals
in $\pint$. For fix projections $P_k := N_{n_k}-N_{m_k}$ where
$\lim_{k\rightarrow+\infty} n_k - m_k =+\infty$ and let $\U$ be a fixed free ultrafilter.
As is well-known we can find an uncountable collection $\Sigma$ of infinite
subsets of $\NN$ with the property that distinct members of $\Sigma$ intersect
only in finite sets. For $\sigma\in\Sigma$, list the elements of $\sigma$ in order
as $s_k$ and build an ultrafilter epimorphism $\phi_\sigma:\tnn\rightarrow\tnn$
as above, this time employing the intervals $P_{s_k}$ and the ultrafilter $\U$.
Write $\Delta_\sigma(X)$ for the diagonal expectation $\sum_{k\in\sigma} P_kXP_k$.
As before, $\ker\Delta_\sigma\subseteq\ker\phi_\sigma$.
Now for any $\sigma\not=\sigma'$,
$\phi_{\sigma}^{-1}(\I_\infty)\not=\phi_{\sigma'}^{-1}(\I_\infty)$, for otherwise
\[
  \phi_{\sigma}^{-1}(\I_\infty)
      = \phi_{\sigma'}^{-1}(\I_\infty)
      \supseteq \ker\Delta_{\sigma} + \ker\Delta_{\sigma'} + \I_\infty
      = \tnn
\]

We can also exhibit infinite chains of ideals in $\pint$ for since
$\phi^{-1}(\I_\infty) \supsetneq \I_\infty$, the ideals
$\P_k := \phi^{-1}(\phi^{-1}(\cdots\phi^{-1}(\I_\infty)\cdots))$
form a chain of distinct ideals in $\pint$ for any fixed epimorphism
$\phi:\tn\rightarrow\tn$.

\subsection{Some properties of ideals in $\pint$}

Although the ultrafilter epimorphism construction of ideals in $\pint$
is not representative, we can prove some properties which all
ideals in $\pint$ share with the ultrafilter construction. These results
are, however, tightly bound to the case of $\tnn$ 
(especially Proposition~\ref{prim-contains-kernel-expectation-prp})
and it is unclear how they might be extended.

\begin{proposition}\label{prim-contains-kernel-expectation-prp}
Let $\P$ be a primitive ideal of $\tnn$ and suppose $\P\supsetneq\I_\infty$.
Then there is an increasing sequence of integers
$n_k$ such that $\P$ contains
\[
  \{ X\in\tn : (N_{n_k} - N_{n_{k-1}})X(N_{n_k} - N_{n_{k-1}}) = 0 \text{ for all $k$} \}
\]
\end{proposition}

\begin{proof}
Let $\L$ be a maximal left ideal such that $\P$ is the kernel of the left-regular
representation on $\tn/\L$.
By Proposition~\ref{prim-contains-cpts-prop}, $\P$ contains a projection
$P\not\in\I^-_I$. Choose a subsequence of nest projections $N_{n_k}$ such that
\[
  \rk(N_{n_{k+1}} - N_{n_k})P \ge \rk N_{n_k}
\]
for all $k$. We shall show that if
\[
  \S := \{ 
      X\in\tn : (N_{n_{2k+2}} - N_{n_{2k}})X(N_{n_{2k+2}} - N_{n_{2k}}) = 0
          \text{ for all $k$} \}
\]
then $\S\subseteq\P$. By Remark~\ref{prim-ideals-remark}, since $\S$ is a two-sided
ideal of $\tn$, if $\S\subseteq\L$ then $\S\subseteq\P$, so suppose for a contradiction that
$\S\not\subseteq\L$.
By maximalty of $\L$, $\S + \L = \tn$ and so there is an $X\in\S$
such that $I-X\in\L$. Decompose $X$ as $Y_0 + Y_1$ where
\[
  Y_0 := \sum_{k=1}^\infty (N_{n_{2k+1}} - N_{n_{2k-1}})X(N_{n_{2k+1}} - N_{n_{2k-1}})
\]
and $Y_1 := X - Y_0$. Observe that therefore
\begin{equation}\label{prim-contains-kernel-expectation-prp:zero-block}
  (N_{n_{k+2}} - N_{n_k})Y_1(N_{n_{k+2}} - N_{n_k}) = 0
\end{equation}
for all $k$.

Now take fixed arbitrary $M < N < I$ in $\N$ and consider two cases. First, if
$N-M$ does not dominate any $N_{n_{k+1}} - N_{n_k}$ then there must be a 
$k$ such that $N-M \le N_{n_{k+2}} - N_{n_k}$, and so $(N-M)Y_1(N-M) = 0$.
On the other hand if $N-M$ does dominate some $N_{n_{k+1}} - N_{n_k}$, take
$k$ to be the largest possible (which exists since $N<I$) and observe that,
by (\ref{prim-contains-kernel-expectation-prp:zero-block}),
\begin{eqnarray*}
  \rk (N-M)Y_1(N-M) &=& \rk N_{n_k}(N-M)Y_1(N-M)  \\
    &\le& \rk N_{n_k}                             \\
    &\le& \rk (N_{n_{k+1}} - N_{n_k})P            \\
    &\le& \rk (N - M)P
\end{eqnarray*}
It follows that in either case
\[
  \rk (N-M)Y_1(N-M) \le \rk (N - M)P.
\]
Since the right-hand side is infinite if $N=I$, the inequality is valid for all $M<N$
in $\N$. It follows immediately from \cite[Theorem 2.6]{OrrPitts:FaTrOp} that $Y_1$
factors through $P$ as $Y_1 = APB$ for some $A, B\in\tn$, and so $Y_1\in\P\subseteq\L$,
whence $I-Y_0\in\L$. 

However since $X\in\S$ the terms of the sum for $Y_0$ are
\begin{align*}
  (N_{n_{2k+1}} - N_{n_{2k-1}}) & X(N_{n_{2k+1}} - N_{n_{2k-1}})  \\
    &= (N_{n_{2k}} - N_{n_{2k-1}})X(N_{n_{2k+1}} - N_{n_{2k}})
\end{align*}
so that $Y_0$ is nilpotent of order $2$. Thus $I-Y_0$ cannot belong to the proper
left ideal $\L$, which is a contradiction.
\end{proof}

Let $E_i$ ($i\in\NN$) be a set of pairwise orthogonal intervals of $\N$.
For $\sigma\subseteq\NN$ let $P_\sigma := \sum_{i\in\sigma}E_i$ and
$\Delta_\sigma(X) := \sum_{i\in\sigma} E_iXE_i$. For convenience
write $\Delta$ for $\Delta_\NN$. The last result shows that, at least in $\tnn$,
primitive ideals which are not in $\pmin$ must contain $\ker\Delta$ for
suitable $\{E_i\}$. The next two lemmas explore the consequences of a primitive
ideal containing $\ker\Delta$, and hold for general nest algebras.

\begin{lemma}\label{prim-ultrafilter-lemma}
Let $\P$ be a primitive ideal of $\tn$ and suppose $\ker\Delta\subseteq\P$.
Then $\Sigma := \{\sigma\subseteq\NN : \ker\Delta_\sigma\subseteq\P\}$
is an ultrafilter.
\end{lemma}

\begin{proof}
$\Sigma$ itself is non-empty since $\NN\in\Sigma$, and
the sets in $\Sigma$ are non-empty since $\ker\Delta_\emptyset = \tn$.
If $\tau\supseteq\sigma$ and $\sigma\in\Sigma$ then
$\ker\Delta_\tau \subseteq\ker\Delta_\sigma\subseteq\P$ and so $\tau\in\Sigma$.
If $\sigma,\tau\in\Sigma$ then
$\ker\Delta_{\sigma\cap\tau} = \ker\Delta_\sigma + \ker\Delta_\tau\subseteq\P$,
and so $\sigma\cap\tau\in\Sigma$. Thus $\Sigma$ is a filter.

Let $\pi:\tn\rightarrow\L(V)$ be an irreducible representation with $\P=\ker\pi$.
For any $\sigma\subseteq\NN$ and $X\in\tn$, $P_\sigma X - XP_\sigma\in\ker\Delta$
and so $\pi(P_\sigma)$ commutes with $\pi(\tn)$. Thus $\ran(\pi(P_\sigma))$ is
an invariant subspace of $\pi(\tn)$ and so $\pi(P_\sigma) = 0, I$.
Suppose that $P_\sigma = I$ and so $P_{\sigma^c}=0$.
Then for any $X\in\tn$,
\[
  X - \Delta_\sigma(X) - \Delta_{\sigma^c}(X) \in \ker\Delta\subseteq \P
\]
and so 
\[
  \pi(X) = \pi(\Delta_\sigma(X) + \Delta_{\sigma^c}(X))
  = \pi(\Delta_\sigma(X)P_\sigma + \Delta_{\sigma^c}(X)P_{\sigma^c})
  = \pi(\Delta_\sigma(X))
\]
whence $\ker\Delta_\sigma\subseteq\P$ and $\sigma\in\Sigma$.
Likewise, if $P_\sigma=0$, then $\sigma^c\in\Sigma$.
Thus $\Sigma$ is an ultrafilter.
\end{proof}

\begin{lemma}\label{decompose-intervals-prp}
Let $\P$ be a primitive ideal of $\tn$ and suppose $\ker\Delta\subseteq\P$.
Suppose that for each $i$ we can decompose $E_i$ as the sum $E^0_i + E^1_i$ of
intervals of $\N$. Then $\P$ contains one of $\ker\Delta^j$ where
$\Delta^j(X) = \sum_{i=1}^\infty E^j_i X E^j_i$.
\end{lemma}

\begin{proof}
Each $E_i$ is decomposed into the sum of two intervals which share a common
endpoint. Let $\sigma$ be the set of $i$ for which the shared endpoint is
the upper endpoint of $E^0_i$ and the lower endpoint of $E^1_i$. Clearly
$\sigma^c$ is then the set of $i$ for which the upper endpoint of $E^1_i$
equals the lower endpoint of $E^0_i$. By Lemma~\ref{prim-ultrafilter-lemma},
$\P$ contains one of $\ker\Delta_\sigma$, $\ker\Delta_{\sigma^c}$. Without loss
of generality assume $\ker\Delta_\sigma\subseteq\P$. Let
$P := \sum_{i\in\sigma} E^0_i$ and observe that for each $i\in\sigma$ there
is an $N_i\in\N$ such that $E^0_i = N_iE_i$ and $E^1_i = N^\perp_iE_i$, and thus
\[
  \Delta_\sigma(P^\perp XP) = \sum_{i\in\sigma} E^1_i XE^0_i
                    = \sum_{i\in\sigma} E_i N_i^\perp XN_iE_i
                    = 0
\]
If $\pi$ is an irreducible representation with $\ker\pi = \P$ then
$\pi(P^\perp XP) = 0$ and so the range of $\pi(P)$ is an invariant
subspace of $\pi(\tn)$, whence, one of $P, P^\perp\in\P$. If $P\in\P$ then
\[
  \ker\Delta^1 \subseteq \ker\Delta_\sigma + P\tn \subseteq\P
\]
while if $P^\perp\in\P$ then
\[
  \ker\Delta^0 \subseteq \ker\Delta_\sigma + \tn P^\perp \subseteq\P
\]
\end{proof}

\begin{theorem}
Let $\P\in\pint$ in $\tnn$. Then there is a free ultrafilter $\U$ and a sequence of
pairwise orthogonal finite-rank intervals $E_i$ such that
$\lim_{i\in\U}\rk E_i = +\infty$ and $\P$ contains
\[
  \{X\in\tnn : \lim_{i\in\U} \|E_i X E_i\| = 0\}
\]
Moreover, given any decomposition of the $E_i$ as the sums of intervals $E^0_i + E^1_i$,
we can replace $\{E_i\}$ with one of $\{E^0_i\}$ or $\{E^1_i\}$.
\end{theorem}

\begin{proof}
The existence of the intervals follows from
Proposition~\ref{prim-contains-kernel-expectation-prp}.
Let $\U$ be the ultrafilter obtained in Lemma~\ref{prim-ultrafilter-lemma}.
If $\lim_{i\in\U} \|E_i X E_i\| = 0$ then, given $\e>0$, there is a $\sigma\in\U$
such that $\|E_i X E_i\|<\e$ for all $i\in\sigma$. Thus taking
$X' := X - \Delta_\sigma(X)$, we see that $\|X - X'\| = \|\Delta_\sigma(X)\|\le \e$
and that $\Delta_\sigma(X') = 0$, whence $X'\in\P$. Thus $X$ is a limit point of $\P$
and since $\P$ is norm closed, $X\in\P$.

Given a decomposition $E_i = E^0_i + E^1_i$, we know from
Proposition~\ref{decompose-intervals-prp} that one of $\ker\Delta^j$ ($j=0,1$)
is in $\P$. Without loss suppose $\ker\Delta^0\subseteq\P$. Again by
Lemma~\ref{prim-ultrafilter-lemma},
$\U^0 := \{\sigma : \ker \Delta^0_\sigma \subseteq\P\}$ is an ultrafilter.
Now let $\sigma\in\U^0$. Since $\U$ is an ultrafilter, one of $\sigma,\sigma^c\in\U$.
But if $\sigma^c\in\U$ then
\[
  \tnn = \ker\Delta^0_\sigma + \ker\Delta_{\sigma^c} \subseteq\P
\]
which is impossible. Thus $\sigma\in\U$ and so, since $\sigma$ was arbitrary,
$\U^0\subseteq\U$. But $\U_0$ is also an ultrafilter, so in fact
$\U_0 = \U$. Thus we may replace $\{E_i\}$ with $\{E^0_i\}$.

Now it follows that the limit of the ranks
of the intervals must be $+\infty$, for otherwise after finitely many
decompositions we could conclude that $\P\supseteq\R^\infty_\N$ and so
$\P\in\pmax$. Similarly if $\U$ were not free then $\P$ would contain
$\{X : E_{i_0}XE_{i_0} = 0\}$ for some $i_0\in\NN$ and, after finitely many
decompositions if necessary, we would see that $\P\supseteq\I_n$ for some $n$,
again contrary to hypothesis.
\end{proof}

\section*{Acknowledgements}

The author gratefully acknowledges the hospitality of Professor~Tony
Carbery and the University of Edinburgh Mathematics Department.

\bibliography{bibliography}
\bibliographystyle{plain}

\end{document}